\documentclass[11pt]{article}

\usepackage[T1]{fontenc}
\usepackage[utf8]{inputenc}
\usepackage{amsmath,amssymb,amsthm,mathrsfs}
\usepackage{hyperref}
\usepackage[margin=1in]{geometry}

\hypersetup{colorlinks=true,linkcolor=black,citecolor=black,urlcolor=black,
  pdftitle={Closed 4-manifolds foliated by hyperplanes},pdfauthor={},pdfsubject={},pdfkeywords={}}

\newif\ifanonimo
\ifdefined\anonimo\anonimotrue\def\limpio{}\else\anonimofalse\fi
\newif\iflimpio
\ifdefined\limpio\limpiotrue\else\limpiofalse\fi

\theoremstyle{plain}
\newtheorem{classical}{Theorem}

\newtheorem{theorem}{Theorem}[section]
\newtheorem{proposition}[theorem]{Proposition}
\newtheorem{lemma}[theorem]{Lemma}
\newtheorem{corollary}[theorem]{Corollary}

\theoremstyle{definition}

\theoremstyle{remark}
\newtheorem{remark}[theorem]{Remark}

\newcommand{\R}{\mathbb{R}}
\newcommand{\Q}{\mathbb{Q}}
\newcommand{\Z}{\mathbb{Z}}
\newcommand{\s}{\mathbb{S}}
\newcommand{\T}{\mathbb{T}}

\newcommand{\F}{\mathcal{F}}
\newcommand{\G}{\mathcal{G}}
\newcommand{\wt}[1]{\widetilde{#1}}
\newcommand{\Deck}{\mathrm{Deck}}
\newcommand{\id}{\mathrm{Id}}
\newcommand{\doi}[1]{\href{https://doi.org/#1}{doi:\nolinkurl{#1}}}

\let\ERthebibliography\thebibliography
\renewcommand{\thebibliography}[1]{\ERthebibliography{#1}\setlength{\itemsep}{1pt}}

\begin{document}

\title{Closed $4$--Manifolds Foliated by Hyperplanes}
\ifanonimo\author{}\else\author{J.\,M.\ Espinar \and H.\ Rosenberg}\fi
\date{}
\maketitle

\begin{abstract}
We give a direct geometric proof that a closed smooth $4$--manifold carrying a $C^1$
foliation by hyperplanes is diffeomorphic to the $4$--torus. As a consequence, an exotic $4$--torus, 
if there is one, carries no $C^1$ foliation by hyperplanes, although it carries $C^0$ ones. Finally, relying on a construction  by B.~Fayad and K.~Khanin, we show that smooth foliations by hyperplanes of the standard $4$--torus, obtained by suspending commuting circle
diffeomorphisms, need not admit a continuous closed defining $1$--form.
\end{abstract}



\section{Introduction}\label{sec:intro}

A codimension--one foliation of an $n$--manifold will be called a
\emph{foliation by hyperplanes} if all its leaves are diffeomorphic
to $\R^{n-1}$. Basic examples are given on the torus $\T^n=\R^n/\Z^n$, where $n\ge2$. Projecting the hyperplanes of $\R^n$ orthogonal to
a vector $v$, we obtain a foliation of $\T^n$. If the coordinates
of $v$ are linearly independent over $\Q$, no non--zero integer
translation preserves a hyperplane. Thus, every leaf is a copy of
$\R^{n-1}$ dense in $\T^n$, and we call these foliations \emph{linear}.

In 1968, H.~Rosenberg~\cite{Rosenberg68}, in collaboration with
J.~Sondow, identified the $3$--torus as the only closed $3$--manifold carrying
a $C^2$ foliation by planes. Using the foliation,
one first shows that the universal cover is $\R^3$ and that the
fundamental group is free abelian, and then, by J.~Stallings'
theorem~\cite[Theorems~1 and~2]{Stallings62}, one gets a fibration
over the circle with fibre a $2$--torus. Since the fundamental
group is abelian, its monodromy is isotopic to the identity.
For instance, this classification applies to Anosov
diffeomorphisms of closed $3$--manifolds whose stable and unstable
foliations are of class $C^2$~\cite[Corollary~2]{Rosenberg68}.
Shortly afterwards, D.~Tischler~\cite[Theorem~1]{Tischler70} proved
that a closed manifold fibres over the circle whenever it carries
a nowhere--vanishing closed $1$--form.

In dimension four, G.~Joubert and
R.~Moussu~\cite[Th\'eor\`eme~2]{JoubertMoussu70} proved the
following version of Rosenberg's theorem in 1970:

\begin{classical}[Joubert--Moussu]\label{main}
Let $M$ be a closed smooth $4$--manifold carrying a $C^2$
codimension--one foliation whose leaves are all diffeomorphic
to $\R^3$. Then $M$ is diffeomorphic to $\T^4$.
\end{classical}

Since a smooth $3$--manifold homeomorphic to $\R^3$ is diffeomorphic
to it~\cite{Moise52,Munkres60}, we can equivalently ask the leaves
to be homeomorphic to $\R^3$. As in dimension three,
Joubert and Moussu obtain a fibration over the circle and then apply
Stallings' theorem and the work of F.~Waldhausen~\cite{Waldhausen68}.
To do this, they start with the bundle--like metric given by
R.~Sacksteder's theorem~\cite{Sacksteder65}, which is smooth with
respect to a possibly different differentiable atlas on $M$.
H.~Imanishi's theorem~\cite[Theorem~1.2 and
pp.~629--630]{Imanishi74} later made it possible to carry out
this argument on the original smooth structure.

In arbitrary dimension, H.~B.~Lawson~\cite[Theorem~33]{Lawson74}
states that a closed $n$--manifold with a $C^2$ foliation by
hyperplanes is homeomorphic to $\T^n$, attributing the case $n=4$
to Joubert and Moussu~\cite{JoubertMoussu70}. For $C^0$ foliations, T.~Yokoyama~\cite[Appendix,
Theorem~3]{YokoyamaThesis} proved that a closed $n$--manifold,
$n\ge3$, carrying a transversely orientable foliation by
topological hyperplanes is homeomorphic to $\T^n$. When $n\ge4$, he
combines a theorem of T.~Li~\cite[Theorem~1.1]{Li02} with the theory
of G.~Hector and U.~Hirsch~\cite{HectorHirsch87} and the topological
rigidity results of F.~T.~Farrell and L.~E.~Jones~\cite{FarrellJones88}.
Moreover, in dimension four, these rigidity results also use the
topological surgery of M.~H.~Freedman~\cite{FreedmanQuinn}.
See also~\cite{YokoyamaTsuboi}.

In this paper, we will give a direct geometric proof of the following
statement (compare~\cite[Ch.~VIII, Theorem~3.2.6]{HectorHirsch87} and
Remark~\ref{rem:HH}):

\begin{theorem}\label{thm:C1}
Let $M$ be a closed smooth $4$--manifold carrying a $C^1$
codimension--one foliation whose leaves are all diffeomorphic
to $\R^3$. Then $M$ is diffeomorphic to $\T^4$.
\end{theorem}

By averaging along a smooth transverse flow, we obtain a smooth closed
$1$--form on $M$ which is positive on the flow (Lemma~\ref{lem:averaging}).
Note that this form need not define the original foliation. In class $C^2$,
Imanishi's theorem gives a topological conjugacy with a
foliation defined by a closed form (Remark~\ref{rem:C2}). With $C^1$ regularity,
such a conjugacy need not exist, since suspending the commuting
diffeomorphisms of V.~Kleptsyn and A.~Navas~\cite{KleptsynNavas08} gives
$C^{1+\alpha}$ foliations by hyperplanes of $\T^4$ with an exceptional
minimal set (Section~\ref{sec:suspension}). In class $C^\infty$,
the construction sketched by B.~Fayad and
K.~Khanin~\cite[Remark~1]{FayadKhanin} gives foliations by hyperplanes
of the standard $\T^4$ which admit no continuous closed defining form
(Proposition~\ref{prop:fk}).
For these foliations, we will also write down Imanishi's differentiable
structure (Proposition~\ref{prop:flat}).%

As a consequence of Theorem~\ref{thm:C1}, we obtain the following.

\begin{corollary}\label{cor:exotic}
Let $N$ be a closed smooth $4$--manifold that is homeomorphic, but not diffeomorphic, to $\T^4$. Then $N$ admits no $C^1$ foliation by hyperplanes. Moreover, $N$ does not fibre smoothly over the circle, and it carries no nowhere--vanishing closed $1$--form of class $C^1$.
\end{corollary}

It remains an open problem whether such a manifold $N$ exists. If it does, the regularity assumption on the foliation
cannot be dropped. To see this,
let us take a homeomorphism $h:\T^4\to N$ and transport a linear foliation
of $\T^4$, obtaining a $C^0$ foliation of $N$ whose leaves are homeomorphic
to $\R^3$. Hence, in class $C^0$, the smooth classification of foliations
by hyperplanes is equivalent to the non--existence of exotic smooth
structures on $\T^4$ (Section~\ref{sec:regularity}). Moreover, these
transported foliations have no smooth transverse flow
(Section~\ref{sec:concluding}).%

\section{Preliminaries}\label{sec:prelim}

Throughout, we assume that manifolds are connected, and by smooth we
mean $C^\infty$. We write $\cong$ for a diffeomorphism of smooth
manifolds or an isomorphism of groups, and $\simeq$ for a homotopy
equivalence. If we assume only $C^1$ regularity, we write
$\cong_{C^1}$, and we use no symbol for homeomorphisms.
By a $C^0$ foliation we mean one given by a topological foliated
atlas. In this case, ``by hyperplanes'' means that its leaves are
homeomorphic to $\R^{n-1}$.

\subsection{Holonomy}

Consider a codimension--one foliation $\F$ of $M$ and a loop
based at $x$ in a leaf $L$. Each such loop determines a germ of a
local homeomorphism of a transversal through $x$. The resulting
homomorphism from $\pi_1(L,x)$, well defined up to conjugacy,
has as image the holonomy group of $L$, a quotient of
$\pi_1(L,x)$, and the same construction applies in class $C^0$
(see~\cite[Chapter~2]{MoerdijkMrcun}).

\begin{lemma}[{\cite[Ch.~VIII, Theorem~2.2.8]{HectorHirsch87}}]\label{lem:nohol}
Let $\F$ be a codimension--one foliation, of class $C^0$, of a closed manifold $M$, all of whose
leaves are simply connected. Then $\F$ has no holonomy and is transversely orientable.
\end{lemma}


On the other hand, we cannot replace simple connectedness by pairwise
diffeomorphic leaves, since even a transversely orientable foliation
can have all its leaves diffeomorphic and non--trivial
holonomy~\cite[Example~3.2]{YokoyamaTsuboi}.

\subsection{Imanishi's theorem}\label{sec:imanishi}

In class $C^2$, Imanishi's theorem supplies a closed form on the original
smooth manifold:

\begin{theorem}[Imanishi~{\cite[Theorem~1.2]{Imanishi74}}]\label{thm:imanishi}
Let $\F$ be a transversely orientable codimension--one foliation of class $C^r$, $r\ge2$, without
holonomy, of a compact smooth manifold $M$. Then there are a nowhere--vanishing closed $1$--form
$\omega$ of class $C^r$ on $M$ and a homeomorphism $h:M\to M$, isotopic to the identity, that maps each leaf of
$\F$ diffeomorphically onto a leaf of $\ker\omega$.
\end{theorem}

For our purposes, we only need $C^1$ regularity of $\omega$ on
the original smooth manifold, together with leafwise $C^1$ regularity
of $h$. Recall that Imanishi first constructs an intermediate
closed form on a modified differentiable structure and then pulls
it back by a diffeomorphism built from a $C^r$ transverse flow
\cite[p.~625 and pp.~629--630]{Imanishi74}.%

\subsection{The flow of a closed form}\label{sec:flow}

Let $\omega$ be a nowhere--vanishing closed $1$--form of class $C^1$ on a
closed smooth manifold $M$, and let $\G$ denote the foliation defined by
$\ker\omega$. Choosing a smooth Riemannian metric gives a $C^1$ vector
field $X=\omega^\sharp/|\omega|^2$, such that $\omega(X)=1$. Then its flow $\psi_t$ is complete, since
$M$ is compact. Let us follow this flow, as in the work of G.~Reeb
(see~\cite[\S1]{Rosenberg66}). We denote by $\pi:\wt M\to M$ the
universal covering, by $\wt\G$ the lifted foliation, and by $\wt X$
the lift of $X$, whose flow is $\wt\psi_t$. Set
$\wt\omega:=\pi^*\omega$.

\begin{proposition}\label{prop:global-product}
There exists a function $f:\wt M\to\R$ of class $C^2$ such that
$df=\wt\omega=\pi^*\omega$. It is a submersion, and
\[
f(\wt\psi_t(p))=f(p)+t\qquad\text{for every }p\in\wt M,\ t\in\R.
\]
Let $\wt L_0:=f^{-1}(0)$. The map
\[
\Phi:\R\times \wt L_0\longrightarrow \wt M,
\qquad
\Phi(t,x)=\wt\psi_t(x),
\]
is a $C^1$ diffeomorphism. The level set $\wt L_0$ is connected, it is a single leaf of
$\wt\G$, and so is every level set $f^{-1}(t)=\wt\psi_t(\wt L_0)$.
\end{proposition}

\begin{proof}
Let us integrate the closed form $\wt\omega$ along paths starting at a
basepoint to define $f$, which is well defined because $\wt M$ is
simply connected, and since the form is of class $C^1$, it follows that $f$ is of class
$C^2$, with $df$ nowhere zero. Along the lifted flow, we have
\[
\frac{d}{dt}f(\wt\psi_t(p))=\wt\omega\bigl(\wt X(\wt\psi_t(p))\bigr)=\omega(X)\circ\pi\equiv1,
\]
and integrating we obtain the stated identity. Following the flow for
time $-f(y)$ brings any $y\in\wt M$ to $\wt L_0$, so this
level is non--empty, and the inverse is given by
\[
\Phi^{-1}(y)=\bigl(f(y),\,\wt\psi_{-f(y)}(y)\bigr).
\]
These maps are inverse to each other by the identity above and the flow
law, and both are of class $C^1$, since $f$ is of class $C^2$ and the
flow of $\wt X$ is of class $C^1$. Since $\wt M$ is connected, $\wt L_0$ is connected, and as the leaves of $\wt\G$ are
the connected components of the levels of $f$, $\wt L_0$ is
a single leaf. Likewise, every other level is a single leaf, since $f^{-1}(t)=\wt\psi_t(\wt L_0)$.
\end{proof}

Fix a basepoint $x_0\in M$ and a lift $\wt x_0\in\wt M$, and
identify $\Gamma:=\Deck(\wt M/M)$ with $\pi_1(M,x_0)$ in the standard
way (see~\cite[Proposition~1.39]{Hatcher02}).

\begin{lemma}\label{lem:deck-period}
For each $\gamma\in\Gamma$ the function $f\circ\gamma-f$ is a constant $c(\gamma)$, and
\(
c:\Gamma\to(\R,+)
\)
is a homomorphism. Moreover, $c(\gamma)=\int_\beta\omega$ for every loop $\beta$ in
$M$ whose lift starting at a point $p$ ends at $\gamma(p)$, so that $c$ is the period homomorphism
of $\omega$.
\end{lemma}

\begin{proof}
Since $\pi\circ\gamma=\pi$, deck transformations preserve the lifted
form, i.e.
\[
d(f\circ\gamma)=\gamma^*(df)=df.
\]
Thus, $f\circ\gamma-f$ is constant, since $\wt M$ is connected.
For $\gamma_1,\gamma_2\in\Gamma$, we have
\[
f\circ(\gamma_1\gamma_2)
=(f\circ\gamma_1)\circ\gamma_2
=(f+c(\gamma_1))\circ\gamma_2
=f+c(\gamma_2)+c(\gamma_1),
\]
and hence $c$ is additive. Finally, if $\wt\beta$ is a lift of $\beta$
from $p$ to $\gamma(p)$, we obtain
\[
\int_\beta\omega=\int_{\wt\beta}df=f(\gamma(p))-f(p)=c(\gamma).
\]
\end{proof}

\begin{proposition}\label{prop:periods}
Assume that the leaves of $\G$ are simply connected. Then:
\begin{enumerate}
\item[(a)] the projection $\pi$ maps each leaf of $\wt\G$ diffeomorphically onto a leaf of $\G$;
\item[(b)] the period homomorphism $c$ of Lemma~\ref{lem:deck-period} is injective, and $\pi_1(M)\cong\Z^k$ for some $k\ge0$.
\end{enumerate}
\end{proposition}

\begin{proof}
(a) Take a leaf $\wt L$ of $\wt\G$, and denote by $L$ the leaf
containing $\pi(\wt L)$, and
in a foliated chart evenly covered by $\pi$,
each plaque of $L$ is the homeomorphic image of the plaques of $\wt L$
lying over it. Thus, $\pi|_{\wt L}:\wt L\to L$ is a covering.
Since $L$ is simply connected, this covering is one--sheeted, and since
$\pi$ is a local diffeomorphism, the covering is a diffeomorphism.

(b) If $c(\gamma)=0$, then $\gamma$ preserves $\wt L_0$
(Proposition~\ref{prop:global-product}), and by (a), $\pi$ is injective on
this leaf, so $\pi\circ\gamma=\pi$ implies that $\gamma$
fixes it pointwise, and a deck transformation with a fixed point must
be the identity. Thus, $\pi_1(M)$ embeds in $(\R,+)$, and hence it is abelian
and torsion--free. Compactness of $M$ also gives finite generation,
and therefore $\pi_1(M)\cong\Z^k$.
\end{proof}

\subsection{Recognizing the torus}

\begin{lemma}\label{lem:rank}
Let $M$ be a closed aspherical $4$--manifold with $\pi_1(M)\cong\Z^k$. Then $k=4$, $M\simeq\T^4$,
and $M$ is orientable.
\end{lemma}

\begin{proof}
Since $M$ and $\T^k$ are both $K(\Z^k,1)$ spaces, and since $M$ has
the homotopy type of a CW complex~\cite[Corollary~A.12]{Hatcher02},
uniqueness of Eilenberg--MacLane
spaces~\cite[Theorem~1B.8]{Hatcher02} gives $M\simeq\T^k$.
With coefficients in $\Z/2$, the closed $4$--manifold $M$ satisfies
$H_4(M;\Z/2)\cong\Z/2$ and $H_j(M;\Z/2)=0$ for $j>4$~\cite[Theorem~3.26]{Hatcher02}.
On the other hand, by the K\"unneth formula, $H_j(\T^k;\Z/2)$ is
non--zero exactly for $0\le j\le k$. Thus, comparing the highest
non--zero degrees, one gets $k=4$. Finally, we have
\[
H_4(M;\Z)\cong H_4(\T^4;\Z)\cong\Z,
\]
and hence $M$ is orientable~\cite[Theorem~3.26]{Hatcher02}.
\end{proof}

To recognize the torus, we will use the following criterion, which
comes from Rosenberg's argument in dimension
three~\cite[p.~132]{Rosenberg68}. Joubert and
Moussu~\cite[Lemme, pp.~A508--A509]{JoubertMoussu70} carried out
this argument in dimension four by using Waldhausen's theorem.

\begin{proposition}\label{prop:criterion}
Let $N$ be a closed smooth aspherical $4$--manifold with $\pi_1(N)\cong\Z^4$. The following
conditions are equivalent:
\begin{enumerate}
\item[(i)] $N$ is diffeomorphic to $\T^4$;
\item[(ii)] $N$ fibres smoothly over $\s^1$;
\item[(iii)] $N$ carries a nowhere--vanishing closed $1$--form of class $C^1$.
\end{enumerate}
\end{proposition}

\begin{proof}
Taking a coordinate projection $\T^4\to\s^1$ gives
(i)$\Rightarrow$(ii), while a smooth fibration $p:N\to\s^1$
gives the closed form $p^*d\theta$, and hence (ii)$\Rightarrow$(iii).

\emph{(iii)$\Rightarrow$(ii).} Let $\omega$ be a nowhere--vanishing closed
$1$--form of class $C^1$ on $N$, and choose smooth closed forms
$\theta_1,\dots,\theta_4$ with integral periods, whose classes form a basis
of $H^1(N;\R)\cong\R^4$ dual to a basis $\{a_1,\dots,a_4\}$ of
$\pi_1(N)\cong\Z^4$. We write
\[
\omega=\sum_i r_i\theta_i+du,
\]
where the $r_i$ are real and $u$ is a function of class $C^2$, obtained
by integrating the zero--period form
\[
\omega-\sum_i r_i\theta_i
\]
along
paths. Choose rational $r_i'$ close to $r_i$, and a smooth function
$u'$ close to $u$ in the $C^1$ topology, and the smooth closed form
\[
\omega'=\sum_i r_i'\theta_i+du'
\]
is uniformly close to $\omega$, hence nowhere zero, and its periods lie
in $\frac1m\Z$ for some integer $m\ge1$. Note that these periods cannot
all vanish, since an exact form vanishes at a maximum of its primitive,
and thus, after rescaling by a positive rational number, we may assume
that the group of periods is $\Z$, i.e. the form has integral periods generating that group. Consider the map
\[
p:N\longrightarrow\s^1=\R/\Z,\qquad p(x)=\int_{x_0}^x\omega'\pmod{\Z},
\]
which is a well defined smooth submersion of the compact manifold $N$,
and therefore a smooth fibre bundle by Ehresmann's theorem,
as in the proof of Tischler's
theorem~\cite[Theorem~1]{Tischler70}. Since $p_*$ is the period
homomorphism of $\omega'$, it is onto, and connectedness
of the fibre follows from the argument below.

\emph{(ii)$\Rightarrow$(i).} Let $p:N\to\s^1$ be a smooth fibre bundle.
If $p_*:\pi_1(N)\to\Z$ were zero, we could lift $p$ to a submersion
$N\to\R$, which is impossible on a compact manifold. Hence, we have  
$p_*(\pi_1(N))=d\Z$ for some $d\ge1$, and replacing $p$ by its lift to the $d$--fold covering
of $\s^1$, we may assume that $p_*$ is onto. From the exact sequence
\[
\pi_1(N)\to\pi_1(\s^1)\to\pi_0(F)\to\pi_0(N)=0
\]
the fibre $F$ is connected, and the homotopy sequence also gives
\[
0\longrightarrow \pi_1(F)
\longrightarrow \pi_1(N)
\longrightarrow \pi_1(\s^1)\cong\Z
\longrightarrow 0.
\]
The same sequence also gives $\pi_j(F)\cong\pi_j(N)=0$ for every $j\ge2$. Thus, $F$ is a closed
aspherical $3$--manifold such that $\pi_1(F)\cong\Z^3$, and it is also
orientable, since $N$ is orientable by Lemma~\ref{lem:rank} and the
normal bundle of a fibre is trivial. By geometrization
(G.~Perelman~\cite{Perelman1,Perelman2}, see~\cite{KleinerLott08} and
also~\cite[p.~448]{Scott83}), this gives a flat geometry, and hence $F$
is homeomorphic to $\T^3$, and by uniqueness of the smooth structure
in dimension three~\cite{Moise52,Munkres60}, we get $F\cong\T^3$.

We now write $N$ as the mapping torus of a diffeomorphism
$\varphi:\T^3\to\T^3$, and
\[
\pi_1(N)\cong\Z^3\rtimes_{\varphi_*}\Z.
\]
Since this group is abelian, $\varphi_*$ is the identity.
After composing $\varphi$ with a translation of $\T^3$,
which is isotopic to the identity, we may assume that $\varphi$ fixes
a point. Since $\T^3$ is a $K(\Z^3,1)$, $\varphi$
is then homotopic to the
identity~\cite[Proposition~1B.9]{Hatcher02}, and $\T^3$ is
irreducible~\cite[Lemma~1]{Rosenberg68} and sufficiently large, since
$\T^2\times\{0\}$ is incompressible. Thus, by Waldhausen's
theorem~\cite[Theorem~7.1]{Waldhausen68}, we have a topological
isotopy to the identity. The inclusion of the diffeomorphism group
into the homeomorphism group induces an injection on components
\cite[Corollaire~1 du th\'eor\`eme~5]{Cerf59}, \cite[Appendix, p.~605]{Hatcher83}
(see also~\cite[\S8.4]{BaderEtAl26}), and hence $\varphi$ is smoothly
isotopic to the identity, and its mapping torus is diffeomorphic to
$\T^3\times\s^1$, that is, $N\cong\T^4$.
\end{proof}

\section{\texorpdfstring{Foliations of class $C^1$ and $C^0$}{Foliations of class C1 and C0}}\label{sec:regularity}

We will prove Theorem~\ref{thm:C1} and Corollary~\ref{cor:exotic}.
Until the proof of Theorem~\ref{thm:C1}, let $\F$ be a transversely orientable
codimension--one foliation of class $C^1$, without holonomy, of a closed smooth
manifold $M$, and choose a smooth vector field $X$ positively transverse to $\F$,
with flow $\psi_t$, by smoothing a continuous vector field positively transverse
to the continuous hyperplane field $T\F$.

Let $\pi:\wt M\to M$ be the universal covering, and denote by $\wt\F$ and
$\wt\psi_t$ the lifts of $\F$ and $\psi_t$, respectively. Since $M$ is compact,
both $\psi_t$ and $\wt\psi_t$ are complete. By~\cite[Ch.~VIII, Theorem~2.2.1]{HectorHirsch87},
which holds in class $C^0$ and for any transverse foliation defined by a complete
flow, we can choose a leaf $\wt L$ of $\wt\F$ and a homeomorphism $\wt M\to\wt L\times\R$ sending the leaves of $\wt\F$ onto the slices $\wt L\times\{t\}$ and the orbits
of $\wt\psi_t$ onto the lines $\{x\}\times\R$.

Denote its second component by $s:\wt M\to\R$. For every $y\in\wt M$,
the map $t\mapsto s(\wt\psi_t(y))$ is a continuous bijection of $\R$, and is
therefore strictly monotone, and in particular the increment $s(\wt\psi_1(y))-s(y)$ never vanishes,
and its sign, since $\wt M$ is connected, is independent of $y$. Thus, replacing
$s$ by $-s$ if necessary, we may assume that all these maps are increasing.

Since every deck transformation $\gamma$ preserves $\wt\F$ and commutes with
$\wt\psi_t$, we have
\[
s\circ\gamma=\tau(\gamma)\circ s
\]
for some increasing homeomorphism $\tau(\gamma)$ of $\R$. Note that
$\tau(\pi_1(M))$ is finitely generated because $\pi_1(M)$ is, and acts freely
on $\R$~\cite[Ch.~VIII, Theorem~2.2.8]{HectorHirsch87}. Moreover, this action is
not trivial, since otherwise $s$ would descend to a function on $M$ increasing
strictly along the orbits of $X$, which is impossible at a maximum. Hence, by
H\"older's theorem in the form of~\cite[Ch.~VIII, Theorem~3.1.10]{HectorHirsch87},
we obtain a continuous non--decreasing surjection $\chi:\R\to\R$ and a
homomorphism $c:\pi_1(M)\to\R$ such that
\[
\chi\circ\tau(\gamma)=\chi+c(\gamma)
\]
for every $\gamma\in\pi_1(M)$.

\begin{lemma}\label{lem:averaging}
The manifold $M$ carries a smooth closed $1$--form $\Omega$ with $\Omega(X)>0$ everywhere, whose period homomorphism is $c$.
\end{lemma}

\begin{proof}
We first choose a flow time with a positive increment at every
point. Set $\sigma=\chi\circ s$, which satisfies
\[
\sigma\circ\gamma=\sigma+c(\gamma)
\]
and is non--decreasing along the orbits of $\wt\psi_t$. For $T>0$, the increment $\sigma\circ\wt\psi_T-\sigma\ge0$ is invariant under deck transformations. Thus it equals $d_T\circ\pi$, where
$d_T\ge0$ is a continuous function on $M$, non--decreasing in $T$ at each point.

Fix $y\in\wt M$. Since $\chi$ is monotone and onto,
$\chi^{-1}(\sigma(y))$ is a compact interval, and since
\[
s(\wt\psi_T(y))\to+\infty
\]
as $T\to+\infty$, it follows that $s(\wt\psi_T(y))$ eventually lies to the
right of this interval, and thereafter $d_T(\pi(y))>0$. The open
sets $\{d_T>0\}$ increase with $T$ and cover the compact manifold $M$,
and taking a finite subcover and choosing $T$ to be the largest of its times,
monotonicity in $T$ gives $d_T>0$ on $M$, and by compactness we can choose
$\delta>0$ such that $d_T\ge\delta$ on $M$.

Next, replace $\sigma$ by a smooth function keeping the same periods.
Choose a smooth closed $1$--form $\theta$ on $M$ whose period homomorphism is
$c$, and a primitive $g$ of $\pi^*\theta$, so that
\[
g\circ\gamma=g+c(\gamma).
\]
Since $\sigma-g$ is invariant under deck transformations, it descends to $M$,
and approximating this descended function within $\delta/4$ by a smooth function
and adding $g$ back, we obtain a smooth function $\sigma_1$ on $\wt M$ such that
\[
\sigma_1\circ\gamma=\sigma_1+c(\gamma)
\]
and $|\sigma_1-\sigma|<\delta/4$.

Averaging over the same time interval, we obtain the smooth function
\[
\bar\sigma(y)=\frac1T\int_0^T\sigma_1(\wt\psi_t(y))\,dt
\]
which satisfies
\[
\bar\sigma\circ\gamma=\bar\sigma+c(\gamma).
\]
Hence $d\bar\sigma=\pi^*\Omega$ for a smooth closed $1$--form $\Omega$ on $M$ whose period homomorphism is $c$.
Differentiating
\[
\bar\sigma(\wt\psi_r(y))=\frac1T\int_r^{r+T}\sigma_1(\wt\psi_t(y))\,dt
\]
at $r=0$, we get
\[
\Omega(X)\circ\pi=\frac{\sigma_1\circ\wt\psi_T-\sigma_1}{T}\ \ge\
\frac{\sigma\circ\wt\psi_T-\sigma-\delta/2}{T}\ \ge\ \frac{\delta}{2T}>0.\qedhere
\]
\end{proof}

The smooth approximation may vary along the leaves of $\wt\F$.
For the smooth classification, positivity on $X$ is all we use.

\begin{lemma}\label{lem:C1periods}
Assume that the leaves of $\F$ are simply connected. Then $\pi$ maps $\wt L$ homeomorphically onto a leaf of $\F$, the
homomorphism $c$ is injective, and $\pi_1(M)\cong\Z^k$ for some $k\ge0$.
\end{lemma}

\begin{proof}
Restrict $\pi$ to $\wt L$. As in Proposition~\ref{prop:periods}(a),
using foliated charts evenly covered by $\pi$, the restriction is a covering of a leaf
of $\F$, and since this leaf is simply connected, the restriction is a
homeomorphism. To study the kernel of $c$, suppose that $c(\gamma)=0$ and
$\tau(\gamma)\ne\id$. Since $\tau(\gamma)$ has no fixed point, its iterates satisfy
\[
\tau(\gamma)^m(x)\to\pm\infty
\]
as $m\to+\infty$, whereas $\chi(\tau(\gamma)^m(x))=\chi(x)$, which is impossible because $\chi$ is non--decreasing and onto.

Thus $c(\gamma)=0$ implies $\tau(\gamma)=\id$, that is, $s\circ\gamma=s$, and hence $\gamma(\wt L)=\wt L$. By injectivity of $\pi|_{\wt L}$,
$\gamma$ fixes every point of $\wt L$, whence $\gamma=\id$. Hence $c$ is
injective, and $\pi_1(M)\cong c(\pi_1(M))$ is a finitely generated subgroup of $\R$. Therefore $\pi_1(M)\cong\Z^k$.
\end{proof}

\begin{proof}[Proof of Theorem~\ref{thm:C1}]
For a codimension--one foliation $\F$ of class $C^1$ of the closed smooth
$4$--manifold $M$, with leaves diffeomorphic to $\R^3$,
Lemma~\ref{lem:nohol} implies that $\F$ has no holonomy and is transversely orientable,
so the preceding construction applies. By Lemma~\ref{lem:C1periods},
$\wt L$ is homeomorphic to a leaf of $\F$, and from the product description
above $\wt M$ is homeomorphic to $\R^3\times\R=\R^4$. In particular, $M$ is aspherical, with $\pi_1(M)\cong\Z^k$,
and Lemma~\ref{lem:rank} gives $k=4$. Finally, Proposition~\ref{prop:criterion}, applied to the form $\Omega$ of
Lemma~\ref{lem:averaging}, gives $M\cong\T^4$.
\end{proof}

\begin{remark}\label{rem:C2}
In the classical $C^2$ case of Joubert and Moussu~\cite{JoubertMoussu70},
Lemma~\ref{lem:nohol} lets us apply Imanishi's
theorem~\cite[Theorem~1.2]{Imanishi74}, obtaining a nowhere--vanishing closed $1$--form
$\omega$ of class $C^1$ on $M$ with its original smooth structure, and the
leaves of $\ker\omega$ are $C^1$ diffeomorphic to $\R^3$.
Let us follow the flow of a complete $C^1$ vector field $X$ such that $\omega(X)=1$. Since a primitive of $\pi^*\omega$ increases at unit speed
along the lifted flow on the universal cover, we get
\[
\widetilde M\cong_{C^1}\R\times\R^3
\]
by Propositions~\ref{prop:periods}(a) and~\ref{prop:global-product}.
Thus $M$ is aspherical. Moreover, the periods of $\omega$ embed $\pi_1(M)$
in $\R$ (Proposition~\ref{prop:periods}(b)), and compactness gives
$\pi_1(M)\cong\Z^k$. By Lemma~\ref{lem:rank}, $k=4$, and
Proposition~\ref{prop:criterion} applied to $\omega$ yields $M\cong\T^4$.
\end{remark}

\begin{remark}\label{rem:HH}
Let us explain the point in the proof of~\cite[Ch.~VIII, Theorem~3.2.6]{HectorHirsch87}
that leads us to give a direct proof of Theorem~\ref{thm:C1}. For $r=1$, the
authors identify $\wt M$ with $\wt L\times\R$ by a trivialization of class
$C^1$~\cite[Ch.~VIII, p.~200]{HectorHirsch87}. In these coordinates
$(x,u)$, as in the proof of Proposition~3.2.5 \cite[Ch.~VIII, p.~198]{HectorHirsch87},
a deck transformation $\gamma$ acts by
\[
(x,u)\mapsto(\bar\gamma(x),\tau(\gamma)(u)),
\]
where $\tau(\gamma)$ is conjugate to the homeomorphism of the same name above,
so that $c$ does not change. They compare $M$ with the quotient $M'$ of
$\wt L\times\R$ by the maps
\[
\gamma'(x,u)=(\bar\gamma(x),u+c(\gamma)).
\]
For a suitable differentiable structure on $M'$, the horizontal foliation of $M'$ is
defined by a closed $1$--form, which they approximate by a closed $1$--form
$\eta$ with cyclic group of periods. They denote by $G$ the kernel of the
period homomorphism of $\eta$. Each leaf of the lift of $\ker\eta$ to
$\wt L\times\R$ meets each vertical line exactly once and is preserved by
$\gamma'$ for $\gamma\in G$, and then the same foliation and vertical lines
are used as a product structure for the deck transformations in
$G$~\cite[Ch.~VIII, pp.~201--202]{HectorHirsch87}.

We will see that, in the situation of Theorem~\ref{thm:C1}, this would force
every $\tau(\gamma)$ to be a translation. If $\gamma\in G$ preserves the lift
of $\ker\eta$, then $\gamma'^{-1}\gamma$ also preserves it, and this map is
\[
(x,u)\mapsto(x,\kappa(u)),
\]
where
\[
\kappa(u)=\tau(\gamma)(u)-c(\gamma).
\]
For $\beta\in G$, the leaf through $(x,u)$ contains
$\beta'(x,u)=(\bar\beta(x),u+c(\beta))$. Hence the leaf through
$(x,\kappa(u))$ contains both $(\bar\beta(x),\kappa(u+c(\beta)))$
and $\beta'(x,\kappa(u))=(\bar\beta(x),\kappa(u)+c(\beta))$, which lie on
the same vertical line. Thus we have
\[
\kappa(u+c(\beta))=\kappa(u)+c(\beta).
\]
In the proof of Theorem~\ref{thm:C1}, we have $\pi_1(M)\cong\Z^4$ and $c$ is injective. Since $\pi_1(M)/G$ is cyclic, it follows that $G$ has
rank at least three and $c(G)$ is dense in $\R$, and hence $\kappa$ commutes
with every translation, so it is a translation, and so is $\tau(\gamma)$.

To obtain the same conclusion for an arbitrary deck transformation, observe that
$\tau(\gamma)=\id$ implies $c(\gamma)=0$. Hence $\tau(G)$ is isomorphic to
$G$ and is a dense group of translations. Since every $\tau(\gamma)$
commutes with this group, it is also a translation.

Then $u\circ\gamma-u$ is constant for every $\gamma$, and $du$ induces a
continuous closed $1$--form defining $\F$. On the other hand,
the $C^{1+\alpha}$ foliations of $\T^4$ by hyperplanes obtained in
Section~\ref{sec:suspension} below from the diffeomorphisms of Kleptsyn and Navas
admit no such form.
For these foliations this step fails for every choice of the transverse
foliation, of the $C^1$ trivialization and of $\eta$.

In Lemma~\ref{lem:averaging} we do not require preservation of this product
structure. We construct a closed form directly on $M$ that is positive
on the transverse field, which suffices for the argument.
\end{remark}

\begin{proof}[Proof of Corollary~\ref{cor:exotic}]
Recall that a manifold homeomorphic to $\T^4$ is aspherical with fundamental group
$\Z^4$. If $N$ carried a $C^1$
foliation by hyperplanes, Theorem~\ref{thm:C1} would give $N\cong\T^4$.
Moreover, if it fibred smoothly over the circle or carried a nowhere--vanishing
closed $1$--form of class $C^1$, we would obtain the same conclusion by
Proposition~\ref{prop:criterion}.
\end{proof}

For $C^0$ foliations, the smooth classification is equivalent to the absence of exotic
smooth structures on $\T^4$. To see this, suppose that a closed smooth
$4$--manifold carries a foliation by hyperplanes, with leaves homeomorphic to
$\R^3$. By Lemma~\ref{lem:nohol}, the foliation has no holonomy and is
transversely orientable. Thus Yokoyama's
theorem~\cite[Appendix, Theorem~3]{YokoyamaThesis} implies that the manifold is
homeomorphic to $\T^4$. Conversely, let $N$ be a smooth manifold homeomorphic
to $\T^4$. Then any homeomorphism $\T^4\to N$ carries a linear foliation
of $\T^4$ by copies of $\R^3$ to a $C^0$ foliation of $N$ whose leaves are
homeomorphic to $\R^3$.

\section{Suspensions and closed defining forms}\label{sec:suspension}

We now consider whether a closed $1$--form defines the foliation, a question which is not settled by knowing the smooth type of the ambient manifold. For suspensions, we will express this condition in terms of the regularity of a common conjugacy between commuting circle diffeomorphisms and rotations.

Let $f=(f_1,f_2,f_3)$ be a triple of pairwise commuting, orientation--preserving $C^r$ diffeomorphisms of $\s^1=\R/\Z$, with $r\ge1$, and choose lifts $F_i:\R\to\R$. To see that these lifts commute, note that their commutator lifts the identity and hence is an integer translation. Since $F_iF_j$ and $F_jF_i$ are conjugate, they have the same translation number. On the other hand, composing with an integer translation adds that integer to the translation number, and we conclude that the commutator is trivial.

Let us write $F^n=F_1^{n_1}F_2^{n_2}F_3^{n_3}$ for $n := (n_1,n_2,n_3)\in\Z^3$, and let $\rho_i\in\R$ be the translation number of $F_i$, and set $\rho:=(\rho_1,\rho_2,\rho_3)\in\R^3$. Then $\rho_i\bmod1$ is the rotation number of $f_i$ and, since the lifts commute, the translation number of $F^n$ is
\[
n\cdot\rho=n_1\rho_1+n_2\rho_2+n_3\rho_3.
\]

Consider the group $\Gamma_f\cong\Z^4$ of diffeomorphisms of $\R^3\times\R$ generated by
\[
(x,t)\longmapsto(x+e_i,F_i(t)),\quad i=1,2,3,\qquad (x,t)\longmapsto(x,t+1),
\]
where $\{e_1,e_2,e_3\}$ is the standard basis of $\R^3$. Its elements are
\[
(x,t)\mapsto(x+n,F^n(t)+\ell)
\]
with $n\in\Z^3$, $\ell\in\Z$, and the action is free and properly discontinuous. Thus, taking the quotient
\[
E_f=(\R^3\times\R)/\Gamma_f,
\]
we obtain a closed $C^r$ manifold fibring over $\T^3$. When $r$ is finite, we give it a compatible smooth structure, whose choice does not affect the diffeomorphism type~\cite[Theorems~2.2.7 and~2.2.9]{Hirsch76}. Moreover, the hyperplanes $\R^3\times\{t\}$ descend to a $C^r$ foliation $\G_f$ of $E_f$, which we call the suspension of $f$.

\begin{lemma}\label{lem:suspension}
Let $f= (f_1,f_2,f_3)$, $\Gamma_f$ and $E_f$ be as above.
\begin{enumerate}
\item[(a)] Let $\lambda\in C^\infty_c(\R^3)$ be non--negative with $\sum_{n\in\Z^3}\lambda(x-n)=1$, and set
\[
A(x,t)=\sum_{n\in\Z^3}\lambda(x-n)\,F^{-n}(t).
\]
Then
$\Psi(x,t)=(x,A(x,t))$ induces a $C^r$ diffeomorphism from $E_f$ onto $\T^4=\R^4/\Z^4$ commuting
with the projections onto $\T^3$.
\item[(b)] If $1,\rho_1,\rho_2,\rho_3$ are linearly independent over $\Q$, every leaf of $\G_f$ is
diffeomorphic to $\R^3$.
\end{enumerate}
\end{lemma}

\begin{proof}
(a) The weighted average is locally finite, and since the $F^n$ commute with each other and with integer translations, we have
\[
A(x+m,F^m(t))=\sum_n\lambda\bigl(x-(n-m)\bigr)F^{-(n-m)}(t)=A(x,t),\qquad A(x,t+1)=A(x,t)+1.
\]
Moreover,
\[
\partial_tA=\sum_n\lambda(x-n)\,(F^{-n})'(t)>0.
\]
Hence, for each $x$, the map $t\mapsto A(x,t)$ is strictly increasing and, by the second identity, has image $\R$. Thus it is a $C^r$ diffeomorphism, and it follows that $\Psi$ is a $C^r$ diffeomorphism of $\R^4$ which, by the identities above, conjugates $\Gamma_f$ to the group of integer translations of $\R^4$.

(b) Let us look at the stabilizer of the horizontal leaf through $(x,t)$, which determines the identifications in its image. It consists of the elements of $\Gamma_f$ preserving $\R^3\times\{t\}$, and these correspond to the $n\in\Z^3$ for which $f_1^{n_1}f_2^{n_2}f_3^{n_3}$ fixes $t\bmod1$. For $n\ne0$, its rotation number $n\cdot\rho\bmod 1$ is irrational, and hence there is no fixed point. Therefore the stabilizers are trivial, i.e. no identifications occur within a horizontal leaf, and each leaf is a diffeomorphic image of $\R^3\times\{t\}$.
\end{proof}

Using $\Psi$, we can regard $\G_f$ as a $C^r$ foliation of the standard $\T^4$. Its leaves are the 
images of the graphs $\{(x,A(x,s)) : x\in\R^3\}$, where $s\in\R$.

As Imanishi already observes, Denjoy's example gives foliations of class $C^1$
without holonomy with exceptional leaves~\cite[p.~608]{Imanishi74}.
Moreover, one cannot extend the conclusion of Imanishi's theorem from $C^2$ to $C^1$ in
dimension four either,
even on the standard $\T^4$. Indeed, V.~Kleptsyn and A.~Navas~\cite{KleptsynNavas08} proved that, for every $0<\alpha<1/3$ and every $\rho_1,\rho_2,\rho_3$ such that $1,\rho_1,\rho_2,\rho_3$ are linearly independent over $\Q$, there are commuting $C^{1+\alpha}$ circle diffeomorphisms with rotation numbers $\rho_i$, none of which is topologically conjugate to a rotation. Since these diffeomorphisms commute, applying Poincar\'e's theory to the first one gives a common minimal set, which is a Cantor set. The proof of Lemma~\ref{lem:suspension} applies verbatim to diffeomorphisms of class $C^{1+\alpha}$, so their suspension is a $C^{1+\alpha}$ foliation of the standard $\T^4$ with an exceptional minimal set and with all its leaves diffeomorphic to $\R^3$.

Suppose that this foliation is topologically conjugate to one defined by a nowhere--vanishing closed $1$--form $\eta$ of class $C^1$. The leaves of the latter would then be simply connected and, by Proposition~\ref{prop:periods}, the period homomorphism of this form would be injective on $\pi_1\cong\Z^4$. Hence its group of periods $P$ would be dense in $\R$. If we take a primitive $g$ of the lift of $\eta$, the full preimage of a leaf in the universal cover is $g^{-1}(b+P)$ for some $b\in\R$ (Proposition~\ref{prop:global-product} and Lemma~\ref{lem:deck-period}). Since $g$ is open, this preimage is dense. Thus every leaf would be dense, contradicting the presence of an exceptional minimal set, and no such conjugacy exists.
Nor is this foliation defined by a continuous closed $1$--form. Indeed, the second half of the proof of Lemma~\ref{lem:density} below, which uses neither Denjoy's theorem nor $r\ge2$, would give a $C^1$ diffeomorphism of $\s^1$ conjugating each $f_i$ to a rotation, which is impossible.

From now on, let us assume that $1,\rho_1,\rho_2,\rho_3$ are linearly independent over $\Q$ and that $r\ge2$. In particular, $\rho_1$ is irrational, and Denjoy's theorem gives an orientation--preserving homeomorphism $H$ of $\s^1$ conjugating $f_1$ to the rotation $R_{\rho_1}$. Moreover, each $Hf_jH^{-1}$ commutes with the irrational rotation $R_{\rho_1}$, and hence is a rotation, necessarily $R_{\rho_j}$~\cite[\S2]{FayadKhanin}. Fix a lift $\wt H$ of $H$. Then
\[
\wt H(t+1)=\wt H(t)+1,
\]
and $\wt H F_i\wt H^{-1}$ is a translation whose translation number equals that of $F_i$. Therefore,
\[
\wt H(F_i(t))=\wt H(t)+\rho_i.
\]

To relate the regularity of the common conjugacy to a closed defining form, observe that such a form gives a positive transverse density invariant under the commuting circle diffeomorphisms. We can integrate this density and normalize the integral to obtain a coordinate on the circle in which all these diffeomorphisms are rotations.

\begin{lemma}\label{lem:density}
Let $0\le k\le r-1$. The suspension $\G_f$ is defined by a nowhere--vanishing closed $1$--form of
class $C^k$ if and only if $H$ is a diffeomorphism of class $C^{k+1}$. (For $k=0$, a continuous
$1$--form is called closed if it is locally the differential of a $C^1$ function.)
\end{lemma}

\begin{proof}
Suppose first that $H$ is a $C^{k+1}$ diffeomorphism and take the form $d\wt H=\wt H'(t)\,dt$ on $\R^3\times\R$, and observe that it is invariant under $\Gamma_f$, since
\[
\wt H(F^n(t)+\ell)=\wt H(t)+n\cdot\rho+\ell.
\]
This form is closed, of class $C^k$ and nowhere zero, with kernel the horizontal foliation, and so it descends to the required form (compare~\cite[p.~628]{Imanishi74}).

Conversely, lift a closed defining form of class $C^k$ to a form $a(x,t)\,dt$ on $\R^3\times\R$, where $a$ is again of class $C^k$. Reversing its sign if necessary, we may assume $a>0$. Locally the form is $du$ for some $C^1$ function $u$, and since the kernel is horizontal,
\[
\partial u/\partial x_j=0,
\]
i.e. $u$ depends locally only on $t$, and by connectedness of $\R^3$, it follows that $a$ depends only on $t$. Invariance under $\Gamma_f$ gives
\[
a(F^n(t)+\ell)\,(F^n)'(t)=a(t).
\]
Normalize the integral by setting
\[
K(t)=\int_0^ta\,/\int_0^1a.
\]
Then $K$ is a $C^{k+1}$ diffeomorphism of $\R$ such that $K(t+1)=K(t)+1$, and differentiation gives $(K\circ F_i)'=K'$. Hence
\[
K\circ F_i=K+\beta_i
\]
for some constants $\beta_i$. Thus the induced circle diffeomorphism conjugates each $f_i$ to a rotation, necessarily $R_{\rho_i}$ because rotation numbers are invariant under conjugacy. To finish, note that this diffeomorphism and $H$ conjugate $f_1$ to the same irrational rotation $R_{\rho_1}$. They therefore differ by a rotation, and we conclude that $H$ is a $C^{k+1}$ diffeomorphism.
\end{proof}

Following Fayad and Khanin~\cite{FayadKhanin}, we call $\rho$ \emph{simultaneously Diophantine} if there are $C,\nu>0$ such that
\[
\max_i\|q\rho_i\|\ge C|q|^{-\nu}
\]
for every integer $q\ne0$, where $\|\cdot\|$ denotes the distance to the nearest integer. They prove that commuting $C^\infty$ circle diffeomorphisms whose rotation numbers are simultaneously Diophantine are $C^\infty$ conjugate to rotations~\cite[Theorem~1]{FayadKhanin}. Moreover, they assert, sketching the construction, that this condition is sharp: for every $\rho$ which is not simultaneously Diophantine there are commuting $C^\infty$ diffeomorphisms with rotation numbers $\rho_i$ whose conjugacy to the rotations is not absolutely continuous~\cite[Remark~1]{FayadKhanin}. We will use this construction to obtain the following examples; only Proposition~\ref{prop:fk} and the comments on its examples depend on it. By Lemma~\ref{lem:suspension} we can identify the ambient smooth manifold, and Lemma~\ref{lem:density} rules out a closed defining form.

\begin{proposition}\label{prop:fk}
There are commuting $C^\infty$ diffeomorphisms $f_1,f_2,f_3$ of $\s^1$ such that the suspension
$\G_f$ is a $C^\infty$ foliation of $\T^4$ whose leaves are all diffeomorphic to $\R^3$, and which
is not defined by any continuous closed $1$--form.
\end{proposition}

\begin{proof}
Choose $\rho$ such that $1,\rho_1,\rho_2,\rho_3$ are linearly independent over $\Q$ and $\rho$ is not simultaneously Diophantine. Both requirements can be satisfied by a Baire argument. Indeed, for integers $m,q_0\ge1$, consider the open set of $\rho\in\R^3$ for which
\[
\max_i\|q\rho_i\|<q^{-m}
\]
for some integer $q\ge q_0$. Since this set contains every point with rational coordinates of common denominator at least $q_0$, it is dense, and Baire's theorem gives a non--empty intersection of these open sets with the complement of the countably many rational affine planes in $\R^3$. Every $\rho$ in this intersection has the required properties.

For such a $\rho$, the construction of~\cite[Remark~1]{FayadKhanin} gives commuting $C^\infty$ diffeomorphisms $f_1,f_2,f_3$ of $\s^1$ with rotation numbers $\rho_1,\rho_2,\rho_3$, whose conjugacy $H$ to the rotations is not absolutely continuous. Lemma~\ref{lem:suspension} identifies $\G_f$ as a $C^\infty$ foliation of $\T^4$ by copies of $\R^3$. If there were a continuous closed defining form, we would obtain from Lemma~\ref{lem:density} that $H$ is a $C^1$ diffeomorphism, and hence absolutely continuous, a contradiction.
\end{proof}

In dimension two, the suspension of an analytic circle diffeomorphism due to V.~I.~Arnold which is not $C^1$ conjugate to a rotation is a foliation of $\T^2$ without holonomy, but no closed $1$--form defines it~\cite[Ch.~VIII, p.~157]{HectorHirsch87}. By Imanishi's theorem, the foliations of Proposition~\ref{prop:fk}, on the standard $\T^4$, are nonetheless carried by a homeomorphism isotopic to the identity onto foliations defined by closed forms.

\subsection{Imanishi's differentiable structure}\label{sec:atlas}

We use the conjugacy to rotations to write the transverse coordinate of Imanishi's construction for these suspensions. Recall that he reparametrizes a transverse flow to carry leaves to leaves, taking its time as transverse coordinate~\cite[pp.~624--625 and Proposition~5.1]{Imanishi74}. Here we use the vertical flow. We will identify the smooth torus which we obtain in this way and compare the identity map between the two differentiable structures.

\begin{proposition}\label{prop:flat}
Let the suspension satisfy the assumptions of Lemma~\ref{lem:density}, and set
\[
\Lambda_\rho=\{(n,\,n\cdot\rho+\ell): n\in\Z^3,\ \ell\in\Z\}\subset\R^4.
\]
\begin{enumerate}
\item[(a)] The map $\widehat H(x,t)=(x,\wt H(t))$ induces a homeomorphism from $E_f$ onto the
flat torus $\R^4/\Lambda_\rho$ that is a diffeomorphism on each leaf and carries $\G_f$ onto the
linear foliation by the hyperplanes $\{u=\text{const}\}$, defined by the closed form $du$.
\item[(b)] Let $\overline E_f$ be $E_f$ with the differentiable structure pulled back from $\R^4/\Lambda_\rho$ by $\widehat H$; its charts are $(x,\wt H(t))$. The topological flow
$\phi_s[x,t]=[x,\wt H^{-1}(\wt H(t)+s)]$
is a reparametrization of the flow of $\Psi^*\partial_u$ along its vertical orbits.
It preserves $\G_f$ and has period one on each vertical circle. Imanishi's
construction with this reparametrization gives the structure $\overline E_f$,
which is diffeomorphic to the standard $\T^4$.
\item[(c)] The identity map $E_f\to\overline E_f$ is a $C^1$ diffeomorphism if and only if $H$ is
a $C^1$ diffeomorphism. For the foliations of Proposition~\ref{prop:fk} it is not even $C^1$.
\end{enumerate}
\end{proposition}

\begin{proof}
(a) By the conjugacy relation
\[
\wt H(F^n(t)+\ell)=\wt H(t)+n\cdot\rho+\ell,
\]
we have
\[
\widehat H\bigl(x+n,F^n(t)+\ell\bigr)=\widehat H(x,t)+(n,n\cdot\rho+\ell).
\]
Thus $\widehat H$ conjugates $\Gamma_f$ to the group of translations by $\Lambda_\rho$. Moreover, it carries each hyperplane $\R^3\times\{t\}$ onto $\{u=\wt H(t)\}$ by the identity in $x$.

(b) The same relation shows that $\phi_s$ is well defined on $E_f$, and since it carries leaves to leaves and follows the same oriented vertical orbits as the flow of $\Psi^*\partial_u$, it is a continuous reparametrization of that flow, i.e. only the time along each orbit changes. Now, period one on each vertical circle follows from the identity
\[
\wt H(t+1)=\wt H(t)+1,
\]
and the time from the plaque $\{t=t_0\}$ to the plaque $\{t=t_1\}$ is $\wt H(t_1)-\wt H(t_0)$. 

Consider any reparametrization of the vertical flow which preserves $\G_f$,
and note that, since it carries leaves to leaves, its lifted time between two plaques is
independent of the point chosen in the starting plaque, and this gives an
invariant measure of $f_1$, unique up to scale. Normalizing the vertical circles to have period one, we obtain the charts $(x,\wt H(t))$.

To identify the smooth torus, take the vectors
\[
\{(e_1,\rho_1),(e_2,\rho_2),(e_3,\rho_3),(0,1)\},
\]
which form a basis of $\R^4$ and generate $\Lambda_\rho$. A linear isomorphism of $\R^4$ carrying this basis to the standard basis maps $\Lambda_\rho$ onto $\Z^4$ and induces the required diffeomorphism $\R^4/\Lambda_\rho\cong\T^4$. 

(c) Use the charts $(x,t)$ of $E_f$ and $(x,u)$ of $\overline E_f$, in which the identity becomes
\[
(x,t)\mapsto(x,\wt H(t)).
\]
For the foliations of Proposition~\ref{prop:fk}, we know that $H$ is not absolutely continuous, and since every $C^1$ map of the circle is absolutely continuous, the identity cannot have this regularity.
\end{proof}

Both differentiable structures therefore give the standard torus. Note that the possible loss of regularity in part~(c) concerns the identity map between the two atlases.

\section{Concluding remarks}\label{sec:concluding}

It remains open whether $\T^4$ admits an exotic smooth structure
\cite{LinMillerEismeier25}. Let us take a homeomorphism $h:\T^4\to N$ onto a closed smooth manifold, and let $\F$ be the image under $h$ of a linear foliation by hyperplanes. Transporting the linear transverse flow gives a topological transverse flow on $N$. We will see that the existence of a smooth transverse flow for $\F$ would imply $N\cong\T^4$. Here transversality means strict monotonicity of the transverse coordinate in local foliated charts along each flow line.

Let $X$ generate such a smooth flow, and let $\widetilde\psi_t$ be its lift, starting at the identity, to the universal cover $\pi:\widetilde N\to N$. Lifting $h^{-1}$ gives a continuous height function $\sigma:\widetilde N\to\R$ whose level sets are the lifted leaves and such that
\[
\sigma\circ\gamma=\sigma+c(\gamma)
\]
for every deck transformation $\gamma$. By connectedness, we may reverse $X$ so that $\sigma$ is strictly increasing along the lifted flow. Since this flow commutes with deck transformations, the positive continuous increment $\sigma\circ\widetilde\psi_1-\sigma$ descends to $N$, and hence, by compactness, there is $\delta>0$ such that
\[
 \sigma\circ\widetilde\psi_1-\sigma\ge\delta.
\]
Now, applying the period-preserving smoothing and averaging in the proof of Lemma~\ref{lem:averaging}, with $T=1$, we obtain a smooth closed $1$--form $\Omega$ on $N$ with $\Omega(X)>0$, and since $N$ is aspherical with fundamental group $\Z^4$, Proposition~\ref{prop:criterion} gives $N\cong\T^4$.

Therefore these transported foliations cannot admit smooth transverse flows on an exotic $4$--torus.

\ifanonimo\else
\section*{Acknowledgments}

We are grateful to T.~Yokoyama and T.~Tsuboi for their careful reading of an earlier version of
this paper. They found a regularity gap in our use of Sacksteder's theorem, showed us the
topological route~\cite{Li02,HectorHirsch87,FarrellJones88} worked out in Yokoyama's
thesis~\cite{YokoyamaThesis}, and pointed out the example of Fayad and Khanin~\cite{FayadKhanin}.
We also thank \'E.~Ghys, who drew our attention to the Hirsch foliation~\cite{Hirsch75Stable}
and to S.~E.~Goodman's 1977 Dijon construction, and an anonymous referee, whose questions on the
differentiable structure in Imanishi's construction and comments on the exposition improved the
paper.

\section*{Declarations}

\noindent\textbf{Funding.} J.\,M.~Espinar is partially supported by the \emph{Mar\'ia de
Maeztu} Excellence Unit IMAG, reference CEX2020-001105-M, funded by
MCIN/AEI/10.13039/501100011033, and by the Spanish grant PID2024-160586NB-I00.

\medskip
\noindent\textbf{Use of AI tools.} During the preparation of this work, the authors used Claude/Code Fable (Anthropic) and ChatGPT/Codex Astra (OpenAI) to assist with the checking of proofs and of references against the sources, with language editing, and with the preparation of the \LaTeX{} code. The authors are responsible for the content of the paper.
\fi

\end{document}